\documentclass{article}

  \usepackage{amssymb}
  \usepackage{amsthm}
  \usepackage{amsmath}
  
\newcommand{\Q}{\mathbb{Q}}
\newcommand{\Qpf}{\Q_{p^f}}

\newcommand{\Fpbar}{\overline{\mathbb{F}}_p}
\newcommand{\fpbar}{\Fpbar}

\newcommand{\F}{\mathbb{F}}
\newcommand{\fp}{\F_p}
\newcommand{\ef}{\F}

\DeclareMathOperator{\GL}{GL}

\theoremstyle{plain} 
\newtheorem{lemma}[equation]{Lemma}
\newtheorem{proposition}[equation]{Proposition}
\newtheorem{theorem}[equation]{Theorem}
\newtheorem{corollary}[equation]{Corollary}
\newtheorem*{axiom}{Axiom}

\theoremstyle{definition}
\newtheorem{definition}[equation]{Definition}

\theoremstyle{remark}
\newtheorem{remark}[equation]{Remark}

\title{On weight elimination for $\GL_n(\Qpf)$}
\author{John Enns}
\date{}

\begin{document}
\maketitle

\begin{abstract}
We show that the modular Serre weights of a sufficiently generic mod $p$ Galois representation of an unramified $p$-adic field are themselves generic, and give precise bounds on the genericity, by extending previous work of Emerton, Gee and Herzig. Our bounds are nearly optimal in some cases. We use this to improve recent weight elimination theorems.
\end{abstract}

Let $K/\mathbb{Q}_{p}$ be a finite extension with residue field $k$. Given a continuous Galois representation $\bar{\rho}:G_{K}\rightarrow\mathrm{GL}_{n}(\fpbar)$, the mod $p$ Langlands program hopes to associate with it in a natural way a smooth admissible $\fpbar$-representation of $\mathrm{GL}_{n}(K)$, denoted $\Pi(\bar{\rho})$ (or possibly a family of such representations). In particular, associated with $\bar{\rho}$ there should be the set $W_{\mathrm{conj}}(\bar{\rho})$ of isomorphism classes of irreducible $\fpbar$-representations of $\mathrm{GL}_{n}(k)$ consisting of the finitely many constituents of the $\mathrm{GL}_{n}(\mathcal{O}_{K})$-socle of $\Pi(\bar{\rho})$. These are known as the ``Serre weights'' of $\bar{\rho}$. Because a mod $p$ local Langlands correspondence remains purely hypothetical except for the case $n=2$, $K=\mathbb{Q}_{p}$ (see \cite{Bre10} for an overview), the definition of $W_{\mathrm{conj}}(\bar{\rho})$ remains conjectural. On the other hand, when $K$ is unramified the papers \cite{Her09} and \cite{GHS} define a set of regular weights $W^{?}(\bar{\rho})$, and it is predicted that the set of regular weights in $W_{\mathrm{conj}}(\bar{\rho})$ should be a subset of $W^{?}(\bar{\rho})$, and coincide with $W^{?}(\bar{\rho})$ when $\bar{\rho}$ is semisimple. Here regularity is a genericity condition; see Section \ref{sec:Background} below for the definition. 

When $\bar{\rho}$ is the local part at $p$ of a global automorphic mod $p$ Galois representation $\bar{r}$, we can define a concrete set of weights $W(\bar{\rho})$, called the set of modular Serre weights of $\bar{\rho}$ using automorphic forms attached to $\bar{r}$. This set is what has traditionally been the subject of ``Serre weight conjectures'' in the literature. It a priori depends on the choice of globalization $\bar{r}$, but the expected local-global compatibility between mod $p$ local Langlands and spaces of automorphic forms means that it should in fact be equal to $W_{\mathrm{conj}}(\bar{\rho})$. There has been much recent progress in proving weight elimination theorems in these global situations. These results typically say that $W_{\mathrm{gen}}(\bar{\rho})\subseteq W^{?}(\bar{\rho})$ where $W_{\mathrm{gen}}(\bar{\rho})$ is defined to be the subset of $W(\bar{\rho})$ consisting of modular weights satisfying some fixed genericity condition. (Thus ``weight elimination'' refers to eliminating weights as possibilities for modular weights of $\bar{\rho}$). See for example the main results of \cite{EGH13} and \cite{LLHLM18}.

The purpose of this paper is to provide a method for improving these results by a priori eliminating all non-generic weights when $K=\mathbb{Q}_{p^{f}}$ is unramified and $\bar{\rho}$ is generic, thus allowing one to replace $W_{\mathrm{gen}}(\bar{\rho})$ by $W(\bar{\rho})$ in the aforementioned weight elimination theorems. Our main result (Theorem \ref{thm:main wt elim}) states that if $\bar{\rho}$ is sufficiently generic then all weights in $W(\bar{\rho})$ are also sufficiently generic, and gives precise bounds. In fact the bounds we give are nearly optimal in a precise sense when $f=1$ (see Remark \ref{rem:optimality}). The proof involves a generalization of techniques from \cite{EGH13}. Roughly speaking the idea is to relate modularity of Serre weights to the existence of lifts of $\bar{\rho}$ with prescribed $p$-adic Hodge-theoretic properties. In fact, we axiomatize the necessary relationship between Serre weights and lifts in Section \ref{sec:Background}, and in Section \ref{sec:Weight-elimination} we prove an abstract weight elimination theorem (Theorem \ref{thm:main wt elim}) assuming this axiom. In Section \ref{sec:Applications} as an application we show that the axiom holds in the global situation considered in \cite{EGH13} and \cite{LLHLM18} and consequently our result improves the main weight elimination theorems in both of those papers. 

\subsection*{Notation}

Throughout we fix a prime $p$. In order for our main result (Theorem \ref{thm:main wt elim}) to not be vacuous, we must eventually assume $p\geq2n+5$, where $n$ is the dimension of the Galois representation in question. We often speak of $\delta$-genericity, whether of Serre weights or Galois representations. Whenever this is mentioned, it is understood that $\delta$ is a positive integer at most $(p-1)/2$. 

If $K/\mathbb{Q}_{p}$ is a finite extension we denote its ring of integers by $\mathcal{O}_{K}$, its residue field by $k$, and its maximal unramified subextension by $K_{0}$. By $\omega_{n}:I_{K}\rightarrow k^{\times}$ we denote the standard character of the inertia group of $K$ of niveau $n\in\mathbb{N}$. Throughout, $E$ is a sufficiently large finite extension of $\mathbb{Q}_{p}$ which serves as the coefficient field for our representations. Its ring of integers is denoted $\mathcal{O}_{E}$ and its residue field by $\ef$. We denote the Teichm\"{u}ller lift $\ef^{\times}\rightarrow\mathcal{O}_{E}^{\times}$ by a tilde $x\mapsto\tilde{x}$. We tacitly fix a choice of embedding $\sigma_{0}:k\hookrightarrow\ef$ and write $\omega_{n}$ also for the composition $I_{K}\rightarrow\ef^{\times}$. None of our results depend on this choice. Let $\bar{\mathbb{Q}}_{p}$ denote a fixed algebraic closure of $E$, having ring of integers $\bar{\mathbb{Z}}_{p}$ and residue field $\fpbar$. We normalize the local reciprocity map $K^{\times}\rightarrow\mathrm{Gal}(\bar{K}/K)^{\mathrm{ab}}$ so that it takes uniformizers to geometric Frobenius elements. We also write $\omega_{n}$ for either of the characters of the same name above thought of as characters of $\mathcal{O}_{K}^{\times}$ via the local reciprocity map. In this way, if $K_{0}$ is degree $f$ then $\omega_{f}:k^{\times}\rightarrow\ef^{\times}$ is just the map induced by the embedding $\sigma_{0}$. 

If $\rho:G_{K}\rightarrow\mathrm{GL}_{n}(E)$ is a potentially semistable representation, we define its associated Weil-Deligne representation $\mathrm{WD}(\rho)$ as in \cite{EGH13}. An inertial type is a representation $I_{K}\rightarrow\mathrm{GL}_{n}(E)$ with open kernel that extends to a representation of the Weil group $W_{K}$. We say that $\rho$ as above has inertial type $\tau$ if $\mathrm{WD}(\rho)|_{I_{K}}=\tau$. We normalize the definition of Hodge-Tate weights so that the cyclotomic character has Hodge-Tate weight $\{-1\}$.

Let $T$ denote the standard diagonal torus in $\mathrm{Res}_{k/\fp}(\mathrm{GL}_{n})$, and $R$ the standard root system with simple roots $\Delta$ corresponding to the Borel subgroup of upper triangular matrices. If $k$ has degree $f$ over $\fp$ we choose an isomorphism $(\mathrm{Res}_{k/\fp}(\mathrm{GL}_{n}))_{\fpbar}\cong\prod_{j=0}^{f-1}\mathrm{GL}_{n,\fpbar}$ by having the $j$th factor correspond to the field embedding $\sigma_{0}^{p^{-j}}:k\hookrightarrow\fpbar$. We write $X(T)$ for the character group of $T$, which we identify with $(\mathbb{Z}^{n})^{f}$ via the previous isomorphism, and define the $p$-restricted weights of $T$ to be $X_{1}(T)=\{\lambda\in X(T)\,|\,\langle\lambda,\alpha^{\vee}\rangle\in[0,p-1]\ \forall\alpha\in\Delta\}$ and $X^{0}(T)=\{\lambda\in X(T)\,|\,\langle\lambda,\alpha^{\vee}\rangle=0\ \forall\alpha\in R^{+}\}$ which are identified with $\{(\lambda_{j})_{j=0}^{f-1}\in(\mathbb{Z}^{n})^{f}\,|\,\lambda_{j}^{i}-\lambda_{j}^{i+1}\in[0,p-1]\}$ and $\{(\lambda_{j})_{j=0}^{f-1}\in(\mathbb{Z}^{n})^{f}\,|\,\lambda_{j}=(a_{j},\ldots,a_{j})\}$ respectively. We let $\pi:X(T)\rightarrow X(T)$ denote the automorphism that sends $(\lambda_{j})_{j}$ to $(\lambda_{j+1})_{j}$. 

If $V$ is a finite length representation then $V^{\mathrm{ss}}$ denotes its semisimplification. If $d\geq1$ we set $e_{d}=p^{d}-1$. We often consider $p$-adic expansions of the form $\sum_{j=0}^{d-1}\alpha_{j}p^{j}$. The indices on the coefficients $\alpha_{j}$ are considered to be defined cyclically modulo $d$, so for example $\alpha_{d}=\alpha_{0}$. 

\section{Preliminaries and the weight elimination axiom}\label{sec:Background}

Let $k/\fp$ be of degree $f$ as in the Notation section. Isomorphism classes of irreducible $\fpbar$-representations of $\mathrm{GL}_{n}(k)$, known informally as weights, are naturally parametrized by $X_{1}(T)/(p-\pi)X^{0}(T)$ (see Lemma 9.2.3 of \cite{GHS}). We write $F(\lambda)$ for the weight corresponding to $\lambda=(\lambda_{j}^{0},\ldots,\lambda_{j}^{n-1})_{j=0}^{f-1}\in X_{1}(T)$. Observe that the differences $\lambda_{j}^{i}-\lambda_{j}^{i+k}$ for $k\geq1$ are independent of the choice of representative $\lambda$, so the following definition makes sense.

\begin{definition}\label{def:gen Serre wts}
Let $F$ be a weight for $\mathrm{GL}_{n}(k)$. We say that $F$ is \emph{regular} if it is of the form $F(\lambda)$ where $\lambda_{j}^{i}-\lambda_{j}^{i+1}\leq p-2$ for each $i,j$. We say that $F$ is \emph{$\delta$-generic} if it is of the form $F(\lambda)$ where $\lambda_{j}^{i}-\lambda_{j}^{i+k}\in[\delta,p-1-\delta]\mod p-1$ for each $i,j$ and $k\geq1$. Note that 1-generic implies regular.
\end{definition}
Most weights are $\delta$-generic, in the sense that if $f,n$ and $\delta$ are fixed, the proportion of weights for $\mathrm{GL}_{n}(k)$ that are $\delta$-generic approaches 1 as $p\rightarrow\infty$. 

We next recall some consequences of the classification of mod $p$ Galois representations. Let $K_{0}$ be the unramified extension of $\mathbb{Q}_{p}$ of degree $f$ (what follows holds without the unramified assumption but our results require it) having residue field $k$. A continuous semisimple representation $\bar{\rho}:G_{K_{0}}\rightarrow\mathrm{GL}_{n}(\fpbar)$ is a direct sum $\bigoplus_{i=0}^{N}\bar{\rho}_{i}$ where $\bar{\rho}_{i}:G_{K_{0}}\rightarrow\mathrm{GL}_{n_{i}}(\fpbar)$ is irreducible; $\sum_{i=0}^{N}n_{i}=n$. Each $\bar{\rho}_{i}|_{I_{K_{0}}}$ will be of the form $\omega_{n_{i}f}^{m_{i}}\oplus\omega_{n_{i}f}^{qm_{i}} \oplus\cdots\oplus\omega_{n_{i}f}^{q^{n_{i}-1}m_{i}}$ where $\frac{q^{n_{i}}-1}{q^{a}-1}\nmid m_{i}$ for any $a|n_{i}$. We may write $m_{i}=\sum_{j=0}^{n_{i}f-1}x_{j}^{i}p^{j}$ where the digits are uniquely specified by stipulating that $x_{j}^{i}\in[0,p-1]$, not all equal to $p-1$. The result of multiplying $m_{i}$ by $q$ is to shift the index $j$ by $f$, so the multiset of digits $\{x_{j}^{i},x_{j+f}^{i},\ldots,x_{j+(n_{i}-1)f}^{i}\}$ is uniquely determined by $\bar{\rho}_{i}$. Hence $\bar{\rho}$ uniquely determines a multiset of $n$ integers 
\begin{equation*}
S_{j}(\bar{\rho})=\{x_{j}^{0},\ldots,x_{j+(n_{0}-1)f}^{0},x_{j}^{1},\ldots,x_{j+(n_{1}-1)f}^{1},\ldots,x_{j}^{N},\ldots,x_{j+(n_{N}-1)f}^{N}\}
\end{equation*}
for each $0\leq j\leq f-1$.

\begin{definition}\label{def:gen gal repns}
If $S=\{s_{0},\ldots,s_{n-1}\}$ is a multiset of integers in $[0,p-1]$, we say that $S$ is \emph{$\delta$-generic} if $s_{a}-s_{b}\in[\delta,p-1-\delta]\mod p-1$ for each $a\neq b$. This implies that the elements of $S$ are distinct if $\delta\geq1$.

Let $\bar{\rho}:G_{K_{0}}\rightarrow\mathrm{GL}_{n}(\fpbar)$ be a continuous representation. We say that $\bar{\rho}$ is \emph{$\delta$-generic} if for each $0\leq j\leq f-1$ the multiset $S_{j}(\bar{\rho}^{\mathrm{ss}})$ is $\delta$-generic. 
\end{definition}
This is simply Definition 3.7 of \cite{LLHLM18} made explicit. As before, most continuous Galois representations are $\delta$-generic.

Next we recall the categories of $p$-adic Hodge theoretic data we need. Let $K=K_{0}((-p)^{1/e_{f}})$.
\begin{definition}\label{def:breuil (sub)modules}
Let $r$ be a nonnegative integer $\leq p-2$. We refer to Section 3.2 of \cite{EGH13} for the definition of a Breuil module over $\mathbb{F}$ of height $r$ with descent data from $K$ to $K_{0}$. These form a category $\mathbb{F}-\mathrm{BrMod}_{\mathrm{dd}}^{r}$. For us, it matters only that an object of this category is in particular a free $(k\otimes_{\fp}\mathbb{F})[u]/u^{ep}$-module $M$ with an action of $\mathrm{Gal}(K/K_{0})$ obeying $g((\alpha\otimes\beta)u^{i}m)=((\omega_{f}(g)^{i}\alpha)\otimes\beta)u^{i}m$ for $\alpha\in k$, $\beta\in\ef$, $g\in\mathrm{Gal}(K/K_{0})$ and $m\in M$, where $\omega_{f}:G_{K_{0}}\rightarrow k^{\times}$ denotes the extension of $\omega_{f}:I_{K_{0}}\rightarrow k^{\times}$ sending $g$ to $g((-p)^{1/e_{f}})/(-p)^{1/e_{f}}\mod p$. 

We refer to Definition 2.2.1 of \cite{HLM17} for the definition of a Breuil submodule of an object $M$ of $\ef-\mathrm{BrMod}_{\mathrm{dd}}^{r}$. Again, for us it only matters that a Breuil submodule of $M$ is in particular a sub-$(k\otimes_{\fp}\ef)[u]/u^{ep}$-module preserved by the action of $\mathrm{Gal}(K/K_{0})$. 

There is a rank-preserving covariant functor $T_{\mathrm{st}}^{r}$ from $\mathbb{F}-\mathrm{BrMod}_{\mathrm{dd}}^{r}$ to the category of continuous $\ef$-representations of $G_{K_{0}}$. See Section 2.2 of \cite{HLM17} and note that the same functor in \cite{EGH13} is written $T_{\mathrm{st}}^{*,r}$.
\end{definition}
The following theorem gathers some of the results about Breuil modules that we need. 

\begin{theorem}\label{thm:existence breuil mods}
Suppose $\rho:G_{K_{0}}\rightarrow\mathrm{GL}_{n}(E)$ is a continuous representation such that $\rho|_{G_{K}}$ is semistable with Hodge-Tate weights contained in $[-r,0]$, and $\rho$ has inertial type $\bigoplus_{i=1}^{n}\chi_{i}$ where $\chi_{i}:I_{K_{0}}\rightarrow\mathcal{O}_{E}^{\times}$. Let $\bar{\rho}$ be the reduction of any $G_{K_{0}}$-stable lattice in $\rho$. Then there is an object $M$ of $\ef-\mathrm{BrMod}_{\mathrm{dd}}^{r}$ with $T_{\mathrm{st}}^{r}(M)\cong\bar{\rho}$ having a $(k\otimes_{\fp}\ef)[u]/u^{ep}$-basis $(v_{i})_{i=1}^{n}$ on which $\mathrm{Gal}(K/K_{0})$ acts as $g\cdot v_{i}=(1\otimes\bar{\chi}_{i}(g))v_{i}$. Moreover, if $\bar{\rho}_{0}\subseteq\bar{\rho}$ is a subrepresentation, then there exists a Breuil submodule $M_{0}\subseteq M$ such that $T_{\mathrm{st}}^{r}(M_{0})\cong\bar{\rho}_{0}$.
\end{theorem}

\begin{proof}
All statements except the last follow from Proposition 3.3.1 of \cite{EGH13}. The last statement follows from Proposition 2.3.5 of \cite{HLM17}.
\end{proof}
In the statement of the theorem, we refer to a Breuil module having a basis as claimed as being of \emph{type} $\bigoplus_{i=1}^{n}\bar{\chi}_{i}$. 

Next we present a general lemma to be used later. The first part is well known, but we lacked a reference for the second part so we have given a proof.
\begin{lemma}\label{lem:semisimple lattice}
Let $E/\mathbb{Q}_{p}$ be a finite extension and $\rho$ a continuous representation of a compact group $G$ on an $n$-dimensional $E$-vector space. Then
\begin{enumerate}
\item There exists a stable $\mathcal{O}_{E}$-lattice in $\rho$, and the semisimplification of the reduction modulo $p$ of any stable $\mathcal{O}_{E}$-lattice is independent of the choice of lattice.
\item After extending coefficients (by a finite amount depending only on $n$) there exists a stable $\mathcal{O}_{E}$-lattice whose reduction modulo $p$ is semisimple.
\end{enumerate}
\end{lemma}

\begin{proof}
It is well-known that the existence of a stable $\mathcal{O}_{E}$-lattice follows from continuity and the compactness of $G$ and the statement about independence of the lattice follows from the Brauer-Nesbitt theorem (see for example Corollary 2.4.8 of \cite{W}). We now prove the second part. Suppose that $\Lambda$ is a $G$-stable lattice and $\bar{\rho}_{0}\subseteq\bar{\rho}:=\Lambda/\varpi_{E}$ is a subrepresentation. Let $e_{1},\ldots,e_{n}$ be an $\mathcal{O}_{E}$-basis of $\Lambda$ which reduces to a basis of $\bar{\rho}$ such that $e_{1},\ldots,e_{k}$ is a basis of $\bar{\rho}_{0}$. Then after extending coefficients to $E'=E(\omega_{E}^{1/2})$, consider the lattice $\Lambda'$ generated by $\varpi_{E}^{-1/2}e_{1},\ldots,\varpi_{E}^{-1/2}e_{k},e_{k+1},\ldots,e_{n}$. One easily sees that $\Lambda'$ is $G$-stable and that $\Lambda'/\varpi_{E'}\cong\bar{\rho}_{0}\oplus\bar{\rho}/\bar{\rho}_{0}$. The claim follows by iterating this argument. 
\end{proof}

Finally we introduce the weight elimination axiom. It will be satisfied in the application we consider in Section \ref{sec:Applications}, where it is a consequence of local-global compatibility at $p$ for automorphic forms on a definite unitary group. It is the crux of the entire method. 

\begin{axiom}[\textbf{WE}]\label{axm}
Let $\bar{\rho}:G_{K_{0}}\rightarrow\mathrm{GL}_{n}(\fpbar)$ be a continuous representation and let $W(\bar{\rho})$ be a set of weights of $\mathrm{GL}_{n}(k)$. We say that $W(\bar{\rho})$ satisfies the weight elimination axiom if whenever $F\in W(\bar{\rho})$ is a Jordan-H\"older constituent of the principal series representation $\mathrm{Ind}_{B_{n}(k)}^{\mathrm{GL}_{n}(k)}(\bigoplus_{i=0}^{n-1}\omega_{f}^{\xi_{i}})$ with $\xi_{i}\in\mathbb{Z}$, $\bar{\rho}$ has a potentially semistable lift $\rho:G_{K_{0}}\rightarrow\mathrm{GL}_{n}(E)$ with Hodge-Tate weights in $[-n+1,0]$ and inertial type $\bigoplus_{i=0}^{n-1}\tilde{\omega}_{f}^{\xi_{i}}$ for some sufficiently large $E$.
\end{axiom}

\begin{remark}
 In fact this axiom is weaker than what holds in our applications. For example, in the case of definite unitary groups in Section \ref{sec:Applications} we in fact know that the lift is potentially crystalline with parallel Hodge-Tate weights $(-n+1,\ldots,-1,0)$. This likely could be used to improve the main result; see Remark \ref{rem:optimality} below.
\end{remark}

\section{Weight elimination}\label{sec:Weight-elimination}

In this section we explain the weight elimination method. From now on assume that $p\geq n+1$. Let $K_{0}$ be as in the previous section.

\begin{proposition}\label{prop:reducing pss repns}
Let $\rho:G_{K_{0}}\rightarrow\mathrm{GL}_{n}(E)$ be a potentially semistable representation with Hodge-Tate weights in $[-n+1,0]$ and inertial type $\bigoplus_{s=0}^{n-1}\tilde{\omega}_{f}^{\xi_{s}}$, and write $\xi_{s}=\sum_{j=0}^{f-1}\xi_{j}^{s}p^{j}$ where $\xi_{j}^{s}\in[0,p-1]$ (not all $=p-1$ for fixed $s$). Let $n=n_{0}+\cdots+n_{N}$ be the partition determined by $\bar{\rho}^{\mathrm{ss}}$. Then
\begin{equation*}
\bar{\rho}^{\mathrm{ss}}|_{I_{K_{0}}}=\bigoplus_{i=0}^{N}\left(\omega_{n_{i}f}^{\theta_{i}} \oplus\omega_{n_{i}f}^{q\theta_{i}} \oplus\cdots\oplus\omega_{n_{i}f}^{q^{n_{i}-1}\theta_{i}}\right)
\end{equation*}
 with $\theta_{i}=\sum_{j=0}^{n_{i}f-1}(\xi_{j}^{s(i,j)}+a_{j}^{i})p^{j}$ where $a_{j}^{i},s(i,j)\in[0,n-1]$.
\end{proposition}

\begin{proof}
We use Theorem \ref{thm:existence breuil mods} along with the description of rank 1 Breuil modules in Lemma 3.3.2 of \cite{EGH13}. Note that $\rho$ becomes semistable over $K_{0}((-p)^{1/e_{f}})$. Choose some $n_{i}$ and consider $\rho^{\mathrm{}}|_{(K_{0})_{n_{i}}}$ where $(K_{0})_{n_{i}}$ denotes the unramified extension of $K_{0}$ of degree $n_{i}$. It becomes semistable over $K_{n_{i}}=(K_{0})_{n_{i}}((-p)^{1/e_{n_{i}f}})$, so by Theorem \ref{thm:existence breuil mods} and Lemma \ref{lem:semisimple lattice} there is an object $M$ of $\ef-\mathrm{BrMod}_{\mathrm{dd}}^{n-1}$ with descent data from $K_{n_{i}}$ to $(K_{0})_{n_{i}}$ having type $\bigoplus_{s=0}^{n-1}\tilde{\omega}_{f}^{\xi_{s}}$ such that $T_{\mathrm{st}}^{n-1}(M)=\bar{\rho}^{\mathrm{ss}}|_{G_{(K_{0})_{n_{i}}}}$. Let $\bar{\rho}_{i}$ be the irreducible summand of $\bar{\rho}^{\mathrm{ss}}$ corresponding to $n_{i}$. By the classification of mod $p$ Galois representations, there is a character $\chi\subseteq\bar{\rho}^{\mathrm{ss}}|_{G_{(K_{0})_{n_{i}}}}$ such that $\bar{\rho}_{i}=\chi\oplus\chi^{q}\oplus\cdots\oplus\chi^{q^{n_{i}-1}}$. Let $N\subseteq M$ be a rank 1 Breuil submodule such that $T_{\mathrm{st}}^{n-1}(N)=\chi$. 

By Lemma 3.3.2 of \cite{EGH13} (the classification of rank 1 Breuil modules) there is a basis $n$ of $N$ such that $g\in\mathrm{Gal}(K_{n_{i}}/(K_{0})_{n_{i}})$ acts via $g\cdot n=\sum_{j=0}^{n_{i}f-1}(\omega_{n_{i}f}(g)^{\kappa_{j}}\otimes1)\epsilon_{j}n$ for some integers $\kappa_{j}$, where $k_{n_{i}}$ denotes the residue field of $K_{n_{i}}$ and $\epsilon_{j}\in k_{n_{i}}\otimes_{\fp}\ef$ is the idempotent corresponding to the embedding $\sigma_{0}^{p^{-j}}:k_{n_{i}}\hookrightarrow\ef$. Moreover, there are integers $r_{j}\in[0,(n-1)e_{n_{i}f}]$ with $r_{j}\equiv p^{n_{i}f-1}\kappa_{j+1}-\kappa_{j}\mod e_{n_{i}f}$ such that 
\begin{equation*}
T_{\mathrm{st}}^{n-1}(N)|_{I_{(K_{0})_{n_{i}}}}=\omega_{n_{i}f}^{\kappa_{0}+\frac{1}{e_{n_{i}f}}\sum_{j=0}^{n_{i}f-1}r_{j}p^{n_{i}f-j}}.
\end{equation*}
As $\ef[\mathrm{Gal}(K_{n_{i}}/(K_{0})_{n_{i}})]$-modules, we have 
\begin{equation*}
\bigoplus_{j=0}^{n_{i}f-1}\omega_{n_{i}f}^{p^{-j}\kappa_{j}}=N/u\hookrightarrow M/u=\left(\bigoplus_{s=0}^{n-1}\omega_{f}^{\xi_{s}}\right)^{\oplus n_{i}f}=\left(\bigoplus_{s=0}^{n-1}\omega_{n_{i}f}^{(1+\cdots+q^{n_{i}-1})\xi_{s}}\right)^{\oplus n_{i}f}.
\end{equation*}
It follows that for each $0\leq j\leq n_{i}f-1$ we have $\kappa_{j}\equiv p^{j}(1+\cdots+q^{n_{i}-1})\xi_{s_{j}}\mod e_{n_{i}f}$ for some $0\leq s_{j}\leq n-1$.

Using this, we calculate that $p^{n_{i}f-1}\kappa_{j+1}-\kappa_{j}\equiv\sum_{t=0}^{n_{i}f-1}(\xi_{t-j}^{s_{j+1}}-\xi_{t-j}^{s_{j}})p^{t}\mod e_{n_{i}f}$. (The lower index on the $\xi_{*}^{*}$ is still defined modulo $f$.) Because for fixed $j$ we have $\xi_{t-j}^{s_{j+1}}-\xi_{t-j}^{s_{j}}\in[-p+1,p-1]$ for all $t$, with not all being $-p+1$ or $p-1$, it follows from the bound on $r_{j}$ that we may write 
\begin{equation*}
r_{j}=\sum_{t=0}^{n_{i}f-1}(\xi_{t-j}^{s_{j+1}}-\xi_{t-j}^{s_{j}})p^{t}+b_{j}e_{n_{i}f}
\end{equation*}
with $b_{j}\in[0,n-1]$. Finally, with some effort we calculate that 
\begin{equation*}
\kappa_{0}+\frac{1}{e_{n_{i}f}}\sum_{j=0}^{n_{i}f-1}r_{j}p^{n_{i}f-j}\equiv\sum_{j=0}^{n_{i}f-1}\xi_{j}^{s_{n_{i}f-j}}p^{j}+\sum_{j=0}^{n_{i}f-1}b_{j}p^{n_{i}f-j}\mod e_{n_{i}f}.
\end{equation*}
Relabelling $a_{j}^{i}=b_{n_{i}f-j}$ and setting $s(i,j)=s_{n_{i}f-j}$, we obtain 
\begin{equation*}
\chi|_{I_{(K_{0})_{n_{i}}}}=\omega_{n_{i}f}^{\sum_{j=0}^{n_{i}f-1}(\xi_{j}^{s(i,j)}+a_{j}^{i})p^{j}}.
\end{equation*}

Repeating this argument for each $0\leq i\leq N$ we arrive at the statement of the proposition.
\end{proof}

We now present the main result.

\begin{theorem}\label{thm:main wt elim}
Let $\bar{\rho}:G_{K_{0}}\rightarrow\mathrm{GL}_{n}(\fpbar)$ be a continuous Galois representation and let $W(\bar{\rho})$ be a set of weights for which Axiom (WE) holds. If $\bar{\rho}$ is $\delta$-generic with $\delta\geq2n+1$ then all weights in $W(\bar{\rho})$ are $(\delta-2n)$-generic. In particular, if $\bar{\rho}$ is $(2n+1)$-generic then all weights in $W(\bar{\rho})$ are regular. If $f=1$ and $\delta\geq n+2$ then all weights in $W(\bar{\rho})$ are $(\delta-(n+1))$-generic.
\end{theorem}

\begin{proof}
Write $\bar{\rho}^{\mathrm{ss}}|_{I_{K_{0}}}= \bigoplus_{i=0}^{N}\omega_{n_{i}f}^{m_{i}}\oplus\omega_{n_{i}f}^{qm_{i}} \oplus\cdots\oplus\omega_{n_{i}f}^{q^{n_{i}-1}m_{i}}$ where $m_{i}=\sum_{j=0}^{n_{i}f-1}x_{j}^{i}p^{j}$ with $x_{j}^{i}\in[0,p-1]$ as in the paragraph preceding Definition \ref{def:gen gal repns}. Suppose that $F(\lambda)\in W(\bar{\rho})$ is a modular weight of $\bar{\rho}$. It follows from Frobenius reciprocity and Lemma 2.3 of \cite{Her11} that $F(\lambda)$ is always a quotient of $\mathrm{Ind}_{B_{n}(k)}^{\mathrm{GL}_{n}(k)}\left(\bigotimes_{s=0}^{n-1}\omega_{f}^{\sum_{j=0}^{f-1}\lambda_{f-j}^{s}p^{j}}\right)$. Write $\sum_{j=0}^{f-1}\lambda_{f-j}^{s}p^{j}\equiv\sum_{j=0}^{f-1}\xi_{j}^{s}p^{j}\mod e_{f}$ with $\xi_{j}^{s}\in[0,p-1]$ (not all equal to $p-1$ for fixed $s$), and observe that $d$-genericity of $F(\lambda)$ for any $d\geq1$ is implied by $(d+n-1)$-genericity of the multisets $\{\xi_{j}^{s}\}^{0\leq s\leq n-1}$ for each fixed $j$. To see this, write $\lambda_{f-j}^{s}=\xi_{j}^{s}+pu_{j}^{s}-u_{j-1}^{s}$ for some integers $\{u_{j}^{s}\}_{0\leq j\leq f-1}^{0\leq s\leq n-1}$ using Lemma \ref{lem:padic} below. Then for $k\geq1$ we have
\begin{equation*}
\lambda_{f-j}^{s}-\lambda_{f-j}^{s+k}=\xi_{j}^{s}-\xi_{j}^{s+k}+p(u_{j}^{s}-u_{j}^{s+k})-(u_{j-1}^{s}-u_{j-1}^{s+k}).
\end{equation*}
The left hand side of this equation lies in $[0,(n-1)(p-1)]$ and $\xi_{j}^{s}-\xi_{j}^{s+k}\in[-p+1,p-1]$, not all equal to $-p+1$ or $p-1$ for fixed $s,k$. We deduce from this that $u_{j}^{s}-u_{j}^{s+k}\in[0,n-1]$ for all $s,j$ and $k\geq1$. It follows that if $\{\xi_{j}^{s}\}^{0\leq s\leq n-1}$ is $(d+n-1)$-generic then for each $s,j$ and $k\geq1$ we have
\begin{equation*}
\lambda_{f-j}^{s}-\lambda_{f-j}^{s+k}\in[d+n-1,p-1-(d+n-1)]+[-n+1,n-1]\mod p-1
\end{equation*}
which gives the claim. We also mention that if $f=1$ then $d$-genericity of $F(\lambda)$ is implied by $d$-genericity of the multiset $\{\xi_{0}^{s}\}^{0\leq s\leq n-1}$ since in this case we automatically have $\lambda_{0}^{s}-\lambda_{0}^{s+k}\equiv\xi_{0}^{s}-\xi_{0}^{s+k}\mod p-1$. 

Now by the weight elimination axiom, $\bar{\rho}$ has a potentially semistable lift with Hodge-Tate weights in $[-n+1,0]$ and inertial type $\bigoplus_{s=0}^{n-1}\tilde{\omega}_{f}^{\sum_{j=0}^{f-1}\xi_{j}^{s}p^{j}}$. From Proposition \ref{prop:reducing pss repns} we deduce an equality 
\begin{equation*}
m_{i}=\sum_{j=0}^{n_{i}f-1}x_{j}^{i}p^{j}\equiv\sum_{j=0}^{n_{i}f-1}(\xi_{j}^{s(i,j)}+a_{j}^{i})p^{j}\mod e_{n_{i}f}
\end{equation*}
where $a_{j}^{i}\in[0,n-1]$, for each $0\leq i\leq N$. We get from Lemma \ref{lem:padic} below that 
\begin{equation*}
\xi_{j}^{s(i,j)}=x_{j}^{i}+t_{j}^{i}p-t_{j-1}^{i}-a_{j}^{i}
\end{equation*}
for some integers $\{t_{j}^{i}\}_{0\leq j\leq n_{i}f-1}^{0\leq i\leq N}$. Because $\xi_{j}^{s(i,j)}+a_{j}^{i}-x_{j}^{i}\in[-p+1,p+n-2]$ for each $i$ and $j$, and are not all $=-p+1$ for a fixed $i$, we must have $t_{j}^{i}\in[0,1]$ (recall that the assumption of $(n+2)$-genericity in particular implies $p\geq2n+5$ here).

We now use these equations to prove genericity. Fix an index $0\leq j\leq f-1$. Let $0\leq i,i'\leq N$, $0\leq d\leq n_{i}-1$ and $0\leq d'\leq n_{i'}-1$ be such that either $i\neq i'$ or else $d\neq d'$. Then
\begin{align*}
\xi_{j}^{s(i,j+df)}-\xi_{j}^{s(i',j+d'f)} &=\xi_{j+df}^{s(i,j+df)}-\xi_{j+d'f}^{s(i',j+d'f)} \\
	& =x_{j+df}^{i}-x_{j+d'f}^{i'}+(t_{j+df}^{i}-t_{j+d'f}^{i'})p\\
	&\quad\quad\quad -(t_{j-1+df}^{i}-t_{j-1+d'f}^{i'})-(a_{j+df}^{i}-a_{j+d'f}^{i'}) \\
	&\in[\delta,p-1-\delta]+[-1,1]\\ 
	&\quad\quad\quad -[-1,1]-[-n+1,n-1]\mod p-1.
\end{align*}
It follows that the multiset $\{\xi_{j}^{s(i,j+df)}\}_{0\leq d\leq n_{i}-1}^{0\leq i\leq N}$ is $(\delta-(n+1))$-generic. In particular since $\delta\geq n+2$, it consists of $n$ distinct elements and therefore must coincide with the multiset $\{\xi_{j}^{s}\}^{0\leq s\leq n-1}$. We conclude that $\{\xi_{j}^{s}\}^{0\leq s\leq n-1}$ is $(\delta-(n+1))$-generic for each $j$ and it follows from the first paragraph that $F(\lambda)$ is $(\delta-2n)$-generic. If $f=1$ then we instead conclude that $F(\lambda)$ is $(\delta-(n+1))$-generic.
\end{proof}

The lemma that follows was used in the above proof.

\begin{lemma}\label{lem:padic}
Let $\alpha_{j}$ be integers, $0\leq j\leq d-1$. We have $\sum_{j=0}^{d-1}\alpha_{j}p^{j}\equiv0\mod e_{d}$ iff $\alpha_{j}=t_{j}p-t_{j-1}$ for some integers $\{t_{j}\}_{j=0}^{d-1}$, and if this holds then $\sum_{j=0}^{d-1}\alpha_{j}p^{j}=t_{d-1}e_{d}$.
\end{lemma}

\begin{proof}
Elementary.
\end{proof}

\begin{remark}\label{rem:optimality}
Theorem \ref{thm:main wt elim} is nearly optimal (at least in the case $f=1$) in the sense that the best possible theorem of this form would give $(\delta-(n-1))$-genericity of the modular weights. This is because one may check that the predicted weight set $W^{?}(\bar{\rho})$ always contains weights whose genericity is $n-1$ less than the genericity of $\bar{\rho}$.

In order to optimize Theorem \ref{thm:main wt elim} one would likely need an improved version of Proposition \ref{prop:reducing pss repns} that gives more information about the $a_{j}^{i}$. For example, using the fact that in practice the Hodge-Tate weights of the lift of $\bar{\rho}$ are exactly $(-n+1,\ldots,-1,0)$ (see the proof of Proposition \ref{prop:axiom holds}), one might be able to show something like $\{a_{j}^{i}\}_{0\leq j\leq n_{i}-1}^{0\leq i\leq N}=\{-n+1,\ldots,-1,0\}$. Then a more careful combinatorial analysis in the proof of Theorem \ref{thm:main wt elim} would give an improved result. We did not pursue this as the relevant $p$-adic Hodge theory results were not readily available.
\end{remark}

\section{Applications}\label{sec:Applications}

In this section we briefly explain how Theorem \ref{thm:main wt elim} applies in the global setting of definite unitary groups of rank $n$ considered in the papers \cite{EGH13, LLHLM18, LLHL}.

We first recall the relevant aspects of the global set-up. The reader is pointed to Section 4.2.2 of \cite{LLHL} (which in turn builds from Section 7.1 of \cite{EGH13}) for more details. Let $F$ be a totally imaginary CM number field with maximal totally real subfield $F^{+}\neq\mathbb{Q}$. Let $c$ be the non-identity element of $\mathrm{Gal}(F/F^{+})$ and assume that all primes of $F^{+}$ dividing $p$ split in $F$ and are unramified. We write $\Sigma_{p}^{+}$ and $\Sigma_{p}$ for the places of $F^{+}$ (resp. $F$) lying above $p$. For each place $v\in\Sigma_{p}^{+}$, choose a place $w\in\Sigma_{p}$ lying above it, so that $v$ splits as $ww^{c}$ in $F$. Write $F_{p}^{+}=F\otimes_{\mathbb{Q}}\mathbb{Q}_{p}$ and $\mathcal{O}_{F^{+},p}=\mathcal{O}_{F^{+}}\otimes_{\mathbb{Z}}\mathbb{Z}_{p}$. 

Let $G$ be a reductive group over $F^{+}$ quasi-split at all finite places, that is an outer form of $\mathrm{GL}_{n}$ and which splits over $F$. Also assume that $G(F_{v}^{+})\cong U_{n}(\mathbb{R})$ at each place $v|\infty$ of $F^{+}$. We fix a choice of model $\mathcal{G}$ of $G$ defined over $\mathcal{O}_{F^{+}}$ such that $\mathcal{G}_{\mathcal{O}_{F_{v}^{+}}}$ is reductive for all places $v$ of $F^{+}$ that split in $F$. For any such place $v$ and $w|v$ we have an isomorphism $i_{w}:G(F_{v}^{+})\xrightarrow{\sim}\mathrm{GL}_{n}(F_{w})$ which restricts to an isomorphism
\begin{equation*}
i_{w}:\mathcal{G}(\mathcal{O}_{F_{v}^{+}})\xrightarrow{\sim}\mathrm{GL}_{n}(\mathcal{O}_{F_{w}}).
\end{equation*}

A \emph{Serre weight for $\mathcal{G}$} is an isomorphism class of an irreducible $\fpbar$-representation of $\mathcal{G}(\mathcal{O}_{F^{+},p})=\prod_{v|p}\mathcal{G}(\mathcal{O}_{F_{v}^{+}})$. Via the (choice of places $w$ and) isomorphisms $i_{w}$ we identify the set of Serre weights for $\mathcal{G}$ with the set of tuples $(F_{v})_{v\in\Sigma_{p}^{+}}$ where each $F_{v}=F(\lambda_{v})$ is a Serre weight for $\mathrm{GL}_{n}(k_{w})$ as in Section \ref{sec:Background} (the tuple $(F_{v})_{v\in\Sigma_{p}^{+}}$ corresponds to the representation $\bigotimes_{v\in\Sigma_{p}^{+}}F(\lambda_{v})\circ i_{w}$). All Serre weights for $\mathcal{G}$ are of this form and two such are equivalent iff the factors $F(\lambda_{v})$ are equivalent for each $v$. We make the obvious definitions in this regard: a Serre weight F for $\mathcal{G}$ is called \emph{regular} or \emph{$\delta$-generic} if each $F_{v}$ is according to the definitions in Section \ref{sec:Background}. These definitions do not depend on the choices of $w$. 

Let $W$ be any finitely generated $\bar{\mathbb{Z}}_{p}$-module with an action of $\mathcal{G}(\mathcal{O}_{F^{+},p})$. If $U\leq G(\mathbb{A}_{F^{+}}^{\infty,p})\times\mathcal{G}(\mathcal{O}_{F^{+},p})$ is a compact open subgroup then the space of algebraic automorphic forms on $G$ with level $U$ and coefficients in $W$ is denoted $S(U,W)$. It is the space of functions $f:G(F^{+})\backslash G(\mathbb{A}_{F^{+}}^{\infty})\rightarrow W$ such that $f(gu)=u_{p}^{-1}\cdot f(g)$ for all $g\in G(\mathbb{A}_{F^{+}}^{\infty})$ and $u\in U$, where $u_{p}$ denotes the image of $u$ under the projection $U\rightarrow\mathcal{G}(\mathcal{O}_{F^{+},p})$. We say that the level $U$ is unramified at a place $v$ which splits in $F$ if $U=\mathcal{G}(\mathcal{O}_{F_{v}^{+}})\times U^{v}$.

 Each space $S(U,W)$ has an action of certain abstract Hecke algebras $\mathbb{T}^{\mathcal{P}}$ consisting of Hecke operators at places in $\mathcal{P}$ where $\mathcal{P}$ is any set of places of $F$ which is a subset of finite complement of the set of places $w'$ of $F$ such that $w'|_{F^{+}}$ splits in $F$, $w'\nmid p$, and $U$ is unramified at $w'|_{F^{+}}$. The details of the definition don't matter so we again refer the reader to \cite{LLHL}. A continuous global Galois representation $\bar{r}:G_{F}\rightarrow\mathrm{GL}_{n}(\fpbar)$ gives rise to a unique maximal ideal $\mathfrak{m}_{\bar{r}}\leq\mathbb{T}^{\mathcal{P}}$ for each $\mathcal{P}$ such that $\bar{r}$ is unramified at all places in $\mathcal{P}$. 

We now fix a continuous irreducible representation $\bar{r}:G_{F}\rightarrow\mathrm{GL}_{n}(\fpbar)$. 
\begin{definition}
 Let $V$ be a Serre weight for $\mathcal{G}$. We say that $\bar{r}$ is automorphic of weight $V$, or that $V$ is a weight of $\bar{r}$, if there exists a compact open $U\leq G(\mathbb{A}_{F^{+}}^{\infty,p})\times\mathcal{G}(\mathcal{O}_{F^{+},p})$ unramified at $p$, and a set $\mathcal{P}$ as above such that $\bar{r}$ is unramified at all places in $\mathcal{P}$, such that $S(U,V)_{\mathfrak{m}_{\bar{r}}}\neq0$. The set of weights of $\bar{r}$ is denoted $W(\bar{r})$. We say that $\bar{r}$ itself is automorphic if $W(\bar{r})\neq\emptyset$. 
\end{definition}

Now we finally make the connection with the method developed in the previous section. Assume the representation $\bar{r}$ is automorphic. From now on, fix a choice of preferred place $v|p$ in $F^{+}$ with residue field $k_{w}$ of cardinality $p^{f}$. Fix any weight $V=\bigotimes_{v'|p}F_{v'}$ of $\bar{r}$ and write $V'=\bigotimes_{v'|p,v'\neq v}F_{v'}$. Then define
\begin{equation*}
W(\bar{r}|_{G_{F_{w}}})=\{\mathrm{Weights\ }F(\lambda)\mathrm{\ of\ }\mathrm{GL}_{n}(k_{w})\,|\,F(\lambda)\otimes V'\in W(\bar{r})\}.
\end{equation*}

\begin{proposition}\label{prop:axiom holds}
The weight set $W(\bar{r}|_{G_{F_{w}}})$ satisfies Axiom (WE). 
\end{proposition}

\begin{proof}
This follows from the proof of Proposition 4.2.5 of \cite{LLHL}, which in fact gives the stronger result that the lift is potentially crystalline with parallel Hodge-Tate weights $(-n+1,\ldots,-1,0)$.
\end{proof}

\begin{corollary}
Theorem 7.5.5 of \cite{EGH13} holds if we replace $W_{\mathrm{gen}}(\bar{r})$ by $W(\bar{r})$ and assume that $\bar{r}|G_{F_{w}}$ is 10-generic for all places $w|p$ of $F$. Theorem 7.5.6 of \cite{EGH13} holds if we replace $W_{\mathrm{gen}}(\bar{r})$ by $W(\bar{r})$ and ``$\bar{r}$ modular of some strongly generic Serre weight'' by ``$\bar{r}$ modular'' if we assume that $\bar{r}|_{G_{F_{w}}}$ is 12-generic for all places $w|p$ of $F$.
\end{corollary}

\begin{proof}
If the Galois representation is 10-generic, then Theorem \ref{thm:main wt elim} says that all weights in $W(\bar{r})$ are 6-generic in our sense. This implies that they are generic in the sense of \cite{EGH13}, so $W_{\mathrm{gen}}(\bar{r})=W(\bar{r})$, which proves the first claim. If the Galois representation is 12-generic then all weights in $W(\bar{r})$ are 8-generic which implies that they are strongly generic in the sense of \cite{EGH13}, and this proves the second claim.
\end{proof}

\begin{corollary}
Under the assumption that $\bar{r}$ is 9-generic at all places above $p$, Theorem 7.8 of \cite{LLHLM18} holds with $W_{\mathrm{elim}}(\bar{r})$ replaced by $W(\bar{r})$ and with the assumption ``$\bar{r}$ automorphic of some reachable Serre weight'' by just ``$\bar{r}$ automorphic''. 
\end{corollary}

\begin{proof}
By Proposition \ref{prop:axiom holds} and Theorem \ref{thm:main wt elim} the assumption of 9-genericity implies that all modular weights are ``reachable'' in the terminology of \cite{LLHLM18}, so $W_{\mathrm{elim}}(\bar{r})=W(\bar{r})$. See Remark 7.10 of \cite{LLHLM18}. 
\end{proof}

\section*{Acknowledgements}
I thank my thesis advisor Florian Herzig for suggesting that the methods of \cite{EGH13} might be used to prove weight elimination results, for encouraging me to write this paper, and for making some helpful comments. I also thank the anonymous referee for helpful comments. 

The author was supported by an NSERC grant while this work was carried out.

\bibliographystyle{amsalpha}
\bibliography{weight_elimination}

\end{document}